\documentclass[12pt,reqno]{amsart}

\usepackage{amsmath}
\usepackage{amssymb}
\usepackage[dvipsnames]{xcolor}
\usepackage[
    colorlinks=true,
    citebordercolor=RoyalBlue,
    linkbordercolor=Red,
    urlbordercolor=RoyalBlue,
    citecolor=red,
    linkcolor=blue,
    urlcolor=blue,
    pdfborder={1 1 1}    
]{hyperref}

\newtheorem{thm}{Theorem}[]

\newtheorem{lem}[thm]{Lemma}
\newtheorem{prop}[thm]{Proposition}
\newtheorem{exam}[thm]{Example}

\newcommand{\brelr}[1]{\left\{#1\right\}}
\newcommand{\n}[1]{\vert#1\vert}
\newcommand{\ngg}[1]{\bigl\vert#1\bigr\vert}

\newcommand{\nl}[1]{\biggl\vert#1\biggr\vert}
\newcommand{\nll}[1]{\Biggl\vert#1\Biggr\vert}
\newcommand{\nlr}[1]{\left\vert#1\right\vert}

\newcommand{\C}{\mathbf{C}}
\newcommand{\N}{\mathbf{N}}

\begin{document}

\title[Borel lemma: geometric progression and zeta-functions]
{\bf Borel lemma: geometric progression and zeta-functions}

\author[Qi Han, Jingbo Liu, and Nadeem Malik]
{Qi Han, Jingbo Liu, and Nadeem Malik}

\address{\rm Department of Computational, Engineering, and Mathematical Sciences
\vskip 2pt Texas A\&M University-San Antonio, San Antonio, Texas 78224, USA
\vskip 2pt Email: {\sf qhan@tamusa.edu}, {\sf jliu@tamusa.edu}, and {\sf nmalik@tamusa.edu}}

\let\origmaketitle\maketitle
\def\maketitle{
  \begingroup
  \def\uppercasenonmath##1{} 
  \let\MakeUppercase\relax 
  \origmaketitle
  \endgroup}






\begin{abstract}
In the proof of the classical Borel lemma \cite{eB} by Hayman \cite{wkH}, each continuous increasing function $T(r)\geq1$ satisfies $T\bigl(r+\frac{1}{T(r)}\bigr)<2T(r)$ outside a possible exceptional set of linear measure $2$.
We note in this work $T(r)$ satisfies a sharper inequality $T\bigl(r+\frac{1}{T(r)}\bigr)<\bigl(\sqrt{T(r)}+1\bigr)^2\leq2T(r)$, if $T(r)\geq\bigl(\sqrt{2}+1\bigr)^2$, outside a possible exceptional set of linear measure $\zeta\bigl(2,\sqrt{2}+1\bigr)\leq0.52<2$ for the Hurwitz zeta-function $\zeta(s,a)$.
This result is worth noting, provided the set of $r$ in which $1\leq T(r)<\bigl(\sqrt{2}+1\bigr)^2$ has linear measure less than $1.48$.
Focusing exclusively on meromorphic functions of infinite order, we utilize Hinkkanen's Second Main Theorem \cite{aH}, draw comparisons with Borel \cite{eB}, Nevanlinna \cite{rN}, and Hayman \cite{wkH}, and finally generalize Fern\'{a}ndez \'{A}rias \cite{aFA1}.
\end{abstract}

\maketitle
\section{\sl Introduction}\label{Int} 
\noindent
The value distribution of a polynomial $p(z)$ in $\C$ is very neat: the Fundamental Theorem of Algebra states that the total number of zeros of $p(z)$, counting multiplicities, equals its degree. Note that $p(z)-c$ is again a polynomial of the same degree for all finite numbers $c$.
Picard generalized this result to entire functions $h(z)$ in $\C$, showing that each $h(z)$ attains all finite values $c$ infinitely often, with at most one exception.
Borel \cite{eB} appears the first to note subtleties regarding the nature of these infinities, and to establish a connection between the growth rate of the maximum modulus of $h(z)$ and the asymptotic frequency with which $h(z)$ attains all finite values $c$, except for at most one.
Building on this, Lindel\"{o}f \cite{elL} employed Jensen's formula in his study of entire functions—a stream of research that markedly influenced Rolf Nevanlinna's development of his value distribution theory for meromorphic functions in $\C$ in the 1920s.

Nevanlinna theory centers around two fundamental results: the First and Second Main Theorems.
While the first result is a novel restatement of the classical Poisson–Jensen formula, the centerpiece of the theory is the second result—a far-reaching, profound quantitative refinement of Picard's theorem on meromorphic functions, which states that each meromorphic function $f(z)$ in $\C$ attains all values—including $\infty$—infinitely often, with at most two exceptions.
Hermann Weyl in 1943 stressed in one of the classical monographs that Nevanlinna theory was ``{\sl one of the few great mathematical events in our century}.''
A fundamental tool in proving the Second Main Theorem is the lemma on the logarithmic derivative (see Han and Liu \cite{HL} and references therein for details), which depends critically on the Poisson–Jensen formula and the Borel lemma.

Interested readers may find Lehto's exposition \cite{oL} particularly enjoyable.
Several classical treatises on Nevanlinna theory include Hayman \cite{wkH}, Gol'dberg and Ostrovski\v{\i} \cite{GO}, and Cherry and Ye \cite{CY}, each offering a distinctive perspective and unique strengths in presenting this elegant theory.

In the sequel, we assume that $f(z)$ is a meromorphic function in $\C$, and define
\begin{equation*}
\begin{split}
m(r,f;\infty)&:=\frac{1}{2\pi}\int_0^{2\pi}\log^+{\ngg{f(re^{i\theta})}}d\theta,\\
m(r,f;c)&:=\frac{1}{2\pi}\int_0^{2\pi}\log^+\frac{1}{\ngg{f(re^{i\theta})-c}}d\theta,\,\,\forall\,\,c\in\C,\\
N(r,f;c)&:=\int_0^r\frac{n(t,f;c)-n(0,f;c)}{t}dt+n(0,f;c)\log r,\,\,\forall\,\,c\in\C\cup\{\infty\},
\end{split}
\end{equation*}
where $n(t,f;c)$ denotes the number of zeros of $f(z)-c$ if $c\in\C$, or the number of poles of $f(z)$ if $c=\infty$, in the Euclidian disk $\n{z}\leq t$.

Nevanlinna's First Main Theorem states that, given the characteristic function
\begin{equation}\label{Equ1}
T(r,f):=m(r,f;\infty)+N(r,f;\infty),
\end{equation}
we have, for each finite value $c\in\C$,
\begin{equation}\label{Equ2}
T(r,f)=m(r,f;c)+N(r,f;c)+\varepsilon(r,c).
\end{equation}
Here, the term $\varepsilon(r,c)$ relies upon $c$ and $r$, and satisfies $\n{\varepsilon(r,c)}=O(1)$ as $r\to\infty$; see, for example, Chapter 1, Theorem 4.1 in Gol'dberg and Ostrovski\v{\i} \cite{GO}.

Moreover, Nevanlinna's Second Main Theorem, presented in its most succinct form to date by Hinkkanen \cite{aH}, states that for $b:=1+\max\limits_{1\leq j\leq q}\{\n{c_j}\}$,
\begin{equation}\label{Equ3}
\begin{split}
&(q-1)T(r,f)-\sum_{j=1}^qN(r,f;c_j)-N(r,f;\infty)+N_{\rm{ram}}(r,f)\\
\leq\,&\frac{1}{2\pi}\Biggl(\int_{\tt{U}}\log^+\nll{\sum_{j=1}^q\frac{f'(re^{i\theta})}{f(re^{i\theta})-c_j}}d\theta+
\int_{\tt{V}}\log^+\nl{\frac{f'(re^{i\theta})}{f(re^{i\theta})-b}}d\theta\Biggr)+K_1\\
\leq\,&\log^+\biggl(\frac{T(R,f)}{R-r}\frac{R}{r}\biggr)+K_2,\,\,\forall\,\,r_0\leq r<R<\infty.
\end{split}
\end{equation}
Here, $c_j\neq c_{j'}\in\C$ for $1\leq j\neq j'\leq q$, $N_{\rm{ram}}(r,f)\geq0$ is the ramification term, ${\tt{U}}$ is the subset of $\theta\in[0,2\pi)$ with $\nl{\sum\limits_{j=1}^q\frac{1}{f(re^{i\theta})-c_j}}>2b+2$ and $f(re^{i\theta})\in\hat{U}$ for the union $\hat{U}$ of the Euclidean disks with centers $c_j$ and radius $\delta\leq\frac{\min\limits_{1\leq j\neq j'\leq q}\brelr{1,\hspace{0.2mm}\nlr{c_j-c_{j'}}}}{3}$, ${\tt{V}}:=[0,2\pi)\setminus{\tt{U}}$, and $K_1,K_2$ and $r_0$ are positive constants.

This sharp formulation of the Second Main Theorem by Hinkkanen was developed partly in response to a question raised by Lang regarding the best possible upper bound for \eqref{Equ3}: Lang \cite{sL} recognized earlier discoveries of Osgood \cite{cfO} and Vojta \cite{pV} on the profound connections between Nevanlinna theory and number theory.
Important contributions to this line of research were also made by Wong \cite{p-mW}, Ye \cite{zY}, and Fern\'{a}ndez \'{A}rias \cite{aFA2}, among others.

Hinkkanen derived an optimal upper bound of the form $\log^+\bigl(\frac{\varphi(T(r,f))}{\mu(r)}\bigr)+O(1)$ for \eqref{Equ3} through increasing functions $\varphi(r),\mu(r)>0$, such that $\int_1^\infty\frac{1}{\varphi(r)}dr<\infty$ and $\int_1^\infty\frac{1}{\mu(r)}dr=\infty$, outside a possible exceptional set of finite $\mu$-measure governed by $\varphi$, revealing an interplay between the magnitude of the upper bound for \eqref{Equ3} and the size of the exceptional set: for instance, when $\mu(r)\equiv1$, a larger $\varphi(r)$ yields a larger upper bound $\log^+\varphi(T(r,f))$, but simultaneously results in a smaller value of $\int_1^\infty\frac{1}{\varphi(r)}dr$, thereby reducing the size of the exceptional set.
Hinkkanen utilized Hayman's version of the Borel lemma (see \cite[Lemma 2.4]{wkH}) by connecting scalar-multiple functions and associated geometric progressions; for additional context, one may also consult Cherry and Ye \cite[Section 3.3]{CY}.

The term
\begin{equation}\label{Equ4}
\log^+\biggl(\frac{T(R,f)}{R-r}\frac{R}{r}\biggr)
\end{equation}
in \eqref{Equ3}, which appears in Hinkkanen \cite[Lemma 3]{aH}, originates from the logarithmic derivative lemma by Gol'dberg and Grin\v{s}te\v{\i}n \cite{GG}.
For its most recent development involving refined, though likely still non-sharp, constants, see Benbourenane and Korhonen \cite{BK}, or Kondratyuk and Kshanovskyy \cite{KK}.
To ensure \eqref{Equ4} be $o(T(r,f))$, the Borel lemma is indispensable, and exceptional sets inevitably appear; refer also to {\sf Remark 1} below.
We will examine this through the Riemann and Hurwitz zeta-functions, aiming to simultaneously reduce both the magnitude of the upper bound for \eqref{Equ4} and the size of the associated exceptional set.

\vskip 8pt
\noindent\fbox{{\sf Remark 1}.} It is well known that when $f(z)$ has finite order, \eqref{Equ4} is $o(T(r,f))$ with no exceptional set, where the order of $f(z)$ is defined to be
\begin{equation}\label{Equ5}
\rho(f):=\limsup_{r\to\infty}\frac{\log T(r,f)}{\log r}.
\end{equation}
Consequently, our subsequent analysis will be confined to meromorphic functions $f(z)$ in $\C$ of infinite order.

\section{\sl Borel lemma: geometric progression and zeta-functions}\label{BL:gp vs. zeta} 
\noindent
We begin this section with an observation concerning the geometric progression
\begin{equation*}
\gamma(s)=\frac{s}{s-1}=\sum_{n=0}^{\infty}\frac{1}{s^n}=1+\frac{1}{s}+\frac{1}{s^2}+\frac{1}{s^3}+\frac{1}{s^4}+\cdots,
\end{equation*}
and the Riemann zeta-function
\begin{equation*}
\zeta(s)=\sum_{n=1}^{\infty}\frac{1}{n^s}=1+\frac{1}{2^s}+\frac{1}{3^s}+\frac{1}{4^s}+\frac{1}{5^s}+\frac{1}{6^s}+\cdots,
\end{equation*}
both restricted to the interval $s\in(1,\infty)$, following Titchmarsh \cite[Section 2.1]{ecT}.

\begin{prop}\label{Titchmarsh}
The inequality $\zeta(s)<\gamma(s)$ holds uniformly for $s\in(1,\infty)$.
\end{prop}

\begin{proof}
Equation (2.1.4) of Titchmarsh \cite{ecT} may be slightly rewritten as
\begin{equation*}
\begin{split}
\zeta(s)\,&=s\int_1^\infty\frac{[t]-t}{t^{s+1}}dt+\frac{s}{2}\int_1^\infty\frac{1}{t^{s+1}}dt+\frac{1}{2}+\frac{1}{s-1}\\
\,&=s\int_1^\infty\frac{[t]-t}{t^{s+1}}dt+1+\frac{1}{s-1}=s\int_1^\infty\frac{[t]-t}{t^{s+1}}dt+\frac{s}{s-1},
\end{split}
\end{equation*}
which can be equivalently reformulated as
\begin{equation}\label{Equ6}
\gamma(s)-\zeta(s)=s\int_1^\infty\frac{t-[t]}{t^{s+1}}dt>0,
\end{equation}
uniformly for $s\in(1,\infty)$, with $[t]$, as usual, the greatest integer part of $t$.
\end{proof}

If we treat $\zeta(1)$ as a {\sl logarithmic infinity} given by $\lim\limits_{n\to\infty}H_n=\lim\limits_{n\to\infty}\sum\limits^n_{k=1}\frac{1}{k}$ for the $n$-th Harmonic number $H_n$ and $\gamma(1)$ as a {\sl linear infinity} given by $\lim\limits_{n\to\infty}\sum\limits^n_{k=1}1$, then $\zeta(s)<\gamma(s)$ uniformly for $s\in[1,\infty)$ with $\lim\limits_{s\to\infty}\gamma(s)=\lim\limits_{s\to\infty}\zeta(s)=1$.

\vskip 8pt
\noindent\fbox{{\sf Remark 2}.} Below, we rederive several renowned forms of the Borel lemma—due to Borel \cite[pp. 374–376]{eB}, Nevanlinna \cite{rN}, and Hayman \cite[Lemma 2.4]{wkH}—deliberately adopting a unified notation to enable clear and direct comparison.
While these reproofs are not entirely new, to the best of our knowledge, this is the first time that they have been presented with a focus on the role of the parameter $s$ in assessing the size of the corresponding exceptional sets.

\begin{thm} {\rm(Borel)}\label{Borel}
Each continuous increasing function $T(r)\geq e$ satisfies
\begin{equation}\label{Equ7}
T\biggl(r+\frac{1}{\log^+T(r)}\biggr)<T^s(r),\,\,r\geq r_0\,\,\text{and}\,\,s>1,
\end{equation}
outside a possible exceptional set of linear measure $\gamma(s)$.
For $T(r)=T(r,f)$ and $R=r+\frac{1}{\log^+T(r,f)}$, this yields the following upper bound for \eqref{Equ4} as
\begin{equation}\label{Equ8}
s\log^+T(r,f)+\log^+\log^+T(r,f)+\log\biggl(1+\frac{1}{r\log^+T(r,f)}\biggr).
\end{equation}
\end{thm}

\begin{proof}
We only sketch the proof following Chapter 3, Theorem 1.2 in Gol'dberg and Ostrovski\v{\i} \cite{GO}.
Assume $T(r)\geq T(r_0)\geq e$ so that $\log T(r_0)\geq1$.
Denote by $E_{\bf{B}}$ the closed subset of $[r_0,\infty)$ in which the inequality \eqref{Equ7} is reversed.

Set $r_1:=\min\limits_{r\geq r_0}\{r\in E_{\bf{B}}\}$, and let $r'_1$ be the least $r$ such that $T(r)=T^s(r_1)$ with $r'_1>r_1$; on the other hand, $T\bigl(r_1+\frac{1}{\log T(r_1)}\bigr)\geq T^s(r_1)$ by definition, and therefore $r_1+\frac{1}{\log T(r_1)}\geq r'_1$ with $r'_1-r_1\leq\frac{1}{\log T(r_1)}$.
Next, set $r_2:=\min\limits_{r\geq r'_1}\{r\in E_{\bf{B}}\}$ and find $r'_2>r_2$ analogously.
Inductively, one derives
\begin{equation*}
r_0\leq r_1<r'_1\leq r_2<r'_2\leq\cdots\leq r_n<r'_n\leq\cdots
\end{equation*}
such that
\begin{equation*}
\begin{split}
r'_1-r_1\leq\,&\frac{1}{\log T(r_1)}\leq\frac{1}{\log T(r_0)}\leq1,\ldots,\\
r'_{n+1}-r_{n+1}\leq\,&\frac{1}{\log T(r_{n+1})}\leq\frac{1}{\log T(r'_n)}\leq\frac{1}{s\log T(r_n)}\\
&\leq\cdots\leq\frac{1}{s^n\log T(r_1)}\leq\frac{1}{s^n\log T(r_0)}\leq\frac{1}{s^n},\ldots.
\end{split}
\end{equation*}
Consequently, $E_{\bf{B}}\subseteq\bigcup^\infty_{j=1}[r_j,r'_j]$, and accordingly $\n{E_{\bf{B}}}\leq\gamma(s)$.
The estimate \eqref{Equ8} follows from a routine computation.
\end{proof}

\begin{thm} {\rm(Nevanlinna)}\label{Nevanlinna}
Each continuous increasing function $T(r)\geq1$ satisfies
\begin{equation}\label{Equ9}
T\biggl(r+\frac{1}{T^s(r)}\biggr)<T(r)+1,\,\,r\geq r_0\,\,\text{and}\,\,s>1,
\end{equation}
outside a possible exceptional set of linear measure $\zeta(s)$.
For $T(r)=T(r,f)$ and $R=r+\frac{1}{T^s(r,f)}$, this yields the following upper bound for \eqref{Equ4} as
\begin{equation}\label{Equ10}
(s+1)\log^+T(r,f)+\log\biggl(1+\frac{1}{T(r,f)}\biggr)+\log\biggl(1+\frac{1}{rT^s(r,f)}\biggr).
\end{equation}
\end{thm}

\begin{proof}
We only sketch the proof following Chapter 3, Theorem 1.2 in Gol'dberg and Ostrovski\v{\i} \cite{GO}, replacing the integrable function $\varphi(r)$ by $\frac{1}{r^s}\,(s>1)$ as well as assuming $T(r)\geq T(r_0)\geq1$.
Denote by $E_{\bf{N}}$ the closed subset of $[r_0,\infty)$ in which the inequality \eqref{Equ9} is reversed.
Likewise, one inductively has
\begin{equation*}
r_0\leq r_1<r'_1\leq r_2<r'_2\leq\cdots\leq r_n<r'_n\leq\cdots
\end{equation*}
such that
\begin{equation*}
\begin{split}
r'_1-r_1\leq\,&\frac{1}{T^s(r_1)}\leq\frac{1}{T^s(r_0)}\leq1,\ldots,\\
r'_n-r_n\leq\,&\frac{1}{T^s(r_n)}\leq\frac{1}{T^s(r'_{n-1})}\leq\frac{1}{{(T(r_{n-1})+1)}^s}\\
&\leq\cdots\leq\frac{1}{{(T(r_1)+n-1)}^s}\leq\frac{1}{{(T(r_0)+n-1)}^s}\leq\frac{1}{n^s},\ldots.
\end{split}
\end{equation*}
As a result, $E_{\bf{N}}\subseteq\bigcup^\infty_{j=1}[r_j,r'_j]$, and accordingly $\n{E_{\bf{N}}}\leq\zeta(s)$.
The estimate \eqref{Equ10} follows from a routine computation.
\end{proof}

It may be worth noting that Nevanlinna \cite{rN} did not explicitly use (or perhaps even consider) $\zeta(s)$, but instead worked with an integral expression that leads to $1+\int_1^\infty\frac{1}{r^s}dr=1+\frac{1}{s-1}=\gamma(s)$ using $\frac{1}{r^s}$ for $s>1$.

\begin{thm} {\rm(Hayman)}\label{Hayman}
Each continuous increasing function $T(r)\geq1$ satisfies
\begin{equation}\label{Equ11}
T\biggl(r+\frac{1}{T(r)}\biggr)<sT(r),\,\,r\geq r_0\,\,\text{and}\,\,s>1,
\end{equation}
outside a possible exceptional set of linear measure $\gamma(s)$.
For $T(r)=T(r,f)$ and $R=r+\frac{1}{T(r,f)}$, this yields the following upper bound for \eqref{Equ4} as
\begin{equation}\label{Equ12}
2\log^+T(r,f)+\log s+\log\biggl(1+\frac{1}{rT(r,f)}\biggr).
\end{equation}
\end{thm}

\begin{proof}
We only sketch the proof following Lemma 2.4 in Hayman \cite{wkH}, assuming $T(r)\geq T(r_0)\geq1$ as before.
Denote by $E_{\bf{H}}$ the closed subset of $[r_0,\infty)$ in which the inequality \eqref{Equ11} is reversed.
 Likewise, one inductively has
\begin{equation*}
r_0\leq r_1<r'_1\leq r_2<r'_2\leq\cdots\leq r_n<r'_n\leq\cdots
\end{equation*}
such that
\begin{equation*}
\begin{split}
r'_1-r_1\leq\,&\frac{1}{T(r_1)}\leq\frac{1}{T(r_0)}\leq1,\ldots,\\
r'_{n+1}-r_{n+1}\leq\,&\frac{1}{T(r_{n+1})}\leq\frac{1}{T(r'_n)}\leq\frac{1}{sT(r_n)}\\
&\leq\cdots\leq\frac{1}{s^nT(r_1)}\leq\frac{1}{s^nT(r_0)}\leq\frac{1}{s^n},\ldots.
\end{split}
\end{equation*}
As a result, $E_{\bf{H}}\subseteq\bigcup^\infty_{j=1}[r_j,r'_j]$, and accordingly $\n{E_{\bf{H}}}\leq\gamma(s)$.
The estimate \eqref{Equ12} follows from a routine computation.
\end{proof}

Since we consider only meromorphic functions $f(z)$ of infinite order, the term $\log^+T(r,f)$ dominates asymptotically—for large $T(r,f)$, note $\eqref{Equ8}<\eqref{Equ12}<\eqref{Equ10}$ when $1<s<2$ whereas $\eqref{Equ12}<\eqref{Equ8}<\eqref{Equ10}$ when $s\geq2$. 
Naturally, the smallest possible exceptional set occurs in Theorem \ref{Nevanlinna} for \eqref{Equ10}, and Theorem \ref{Hayman} produces a larger exceptional set for \eqref{Equ12}, while Theorem \ref{Borel} yields the largest exceptional set for \eqref{Equ8} when incorporating the set of $r$ in which $1\leq T(r,f)<e$.

\begin{thm}\label{HanLiu1}
Each continuous increasing function $T(r)\geq1$ satisfies
\begin{equation}\label{Equ13}
T\biggl(r+\frac{1}{T(r)}\biggr)<{\bigl(T^{1/s}(r)+1\bigr)}^s,\,\,r\geq r_0\,\,\text{and}\,\,s>1,
\end{equation}
outside a possible exceptional set of linear measure $\zeta(s)$.
For $T(r)=T(r,f)$ and $R=r+\frac{1}{T(r,f)}$, this yields the following upper bound for \eqref{Equ4} as
\begin{equation}\label{Equ14}
2\log^+T(r,f)+\log\biggl(1+\frac{1}{T^{1/s}(r,f)}\biggr)^s+\log\biggl(1+\frac{1}{rT(r,f)}\biggr).
\end{equation}
\end{thm}

\begin{proof}
Assume $T(r)\geq T(r_0)\geq1$, and denote by $E$ the closed subset of $[r_0,\infty)$ in which the inequality \eqref{Equ13} is reversed.
One inductively deduces
\begin{equation*}
r_0\leq r_1<r'_1\leq r_2<r'_2\leq\cdots\leq r_n<r'_n\leq\cdots
\end{equation*}
such that
\begin{equation*}
\begin{split}
r'_1-r_1\leq\,&\frac{1}{T(r_1)}\leq\frac{1}{T(r_0)}\leq1\,\,\text{and}\,\,T(r_1)\geq T(r_0)\geq1=1^s,\\
r'_2-r_2\leq\,&\frac{1}{T(r_2)}\leq\frac{1}{T(r'_1)}\leq\frac{1}{{(T^{1/s}(r_1)+1)}^s}\leq\frac{1}{2^s}\,\,\text{and}\,\,T(r_2)\geq2^s,\ldots,\\
r'_n-r_n\leq\,&\frac{1}{T(r_n)}\leq\frac{1}{T(r'_{n-1})}\leq\frac{1}{{(T^{1/s}(r_{n-1})+1)}^s}\leq\frac{1}{n^s}\,\,\text{and}\,\,T(r_n)\geq n^s,\ldots.
\end{split}
\end{equation*}
Consequently, $E\subseteq\bigcup^\infty_{j=1}[r_j,r'_j]$, and accordingly $\n{E}\leq\zeta(s)$.
The estimate \eqref{Equ14} follows from a routine computation.
\end{proof}

As the term $\log^+T(r,f)$ dominates asymptotically, $\eqref{Equ8}<\eqref{Equ14}<\eqref{Equ12}<\eqref{Equ10}$ if $1<s<2$ and $\eqref{Equ14}<\eqref{Equ12}<\eqref{Equ8}<\eqref{Equ10}$ if $s\geq2$ for large $T(r,f)$; the smallest possible exceptional set arises in Theorem \ref{Nevanlinna} for \eqref{Equ10} and in Theorem \ref{HanLiu1} for \eqref{Equ14}.
Theorem \ref{HanLiu1}, in particular, suggests the possibility of simultaneously getting both a sharper upper bound for \eqref{Equ4} and a smaller associated exceptional set.

The final point of discussion regards the comparison between \eqref{Equ12} and \eqref{Equ14} in the case where $T(r,f)$ is not necessarily large.
Observe that
\begin{equation*}
\biggl(1+\frac{1}{T^{1/s}(r,f)}\biggr)^s\leq s\,\,\Longleftrightarrow\,\,T(r,f)\geq\Bigl(\frac{1}{s^{1/s}-1}\Bigr)^s.
\end{equation*}
For some $f(z)$ with $\rho(f)=\infty$, \eqref{Equ5} yields $T(r,f)\gg r^\omega,\,\,\forall\,\,\omega>0$.
So, given $s\geq2$ and those $f(z)$, \eqref{Equ14} is sharper than \eqref{Equ12} whenever $r>\bigl(\frac{1}{s^{1/s}-1}\bigr)^{s/\omega}\approx1$.

\vskip 8pt
\noindent\fbox{{\sf Remark 3}.} It is important to realize that in \eqref{Equ12} and \eqref{Equ14}, the dominant term is $2\log^+T(r,f)$, which, unlike in \eqref{Equ8} and \eqref{Equ10}, in independent of $s$.
It is reasonable to speculate Hayman originally derived \eqref{Equ12} for $s=2$, a result we generalized to all $s>1$ in Theorem \ref{Hayman}, with the aim of enhancing Borel and Nevanlinna's works.
In this light, Theorem \ref{HanLiu1} may be seen as a further refinement of Hayman's result, and thus a potential supplement to the classical value distribution theory.

Given {\sf Remark 3} above, we provide a detailed analysis for the case $s=2$ using the Hurwitz zeta-function, which is summarized as the following example.

\begin{exam}\label{s=2}
When $s=2$, one has $\eqref{Equ14}<\eqref{Equ12}$ if $T(r,f)>T(r'_0,f)=\bigl(\sqrt{2}+1\bigr)^2$ for some fixed $r'_0\geq r_0>0$ with $T(r_0,f)\geq1$.
Let $E'$ denote the closed subset of $[r'_0,\infty)$ in which the inequality \eqref{Equ13} fails to hold.
One inductively deduces
\begin{equation*}
r'_0\leq r_1<r'_1\leq r_2<r'_2\leq\cdots\leq r_n<r'_n\leq\cdots
\end{equation*}
following the proof of Theorem \ref{HanLiu1}, by examining $E'$ instead of $E$, such that
\begin{equation*}
\begin{split}
r'_1-r_1\leq\,&\frac{1}{T(r_1,f)}\leq\frac{1}{T(r'_0,f)}\leq\frac{1}{\bigl(\sqrt{2}+1\bigr)^2}\\
&\,\,\text{and}\,\,T(r_1,f)\geq\bigl(\sqrt{2}+1\bigr)^2,\\
r'_2-r_2\leq\,&\frac{1}{T(r_2,f)}\leq\frac{1}{T(r'_1,f)}\leq\frac{1}{{(T^{1/2}(r_1,f)+1)}^2}\leq\frac{1}{\bigl(\sqrt{2}+2\bigr)^2}\\
&\,\,\text{and}\,\,T(r_2,f)\geq\bigl(\sqrt{2}+2\bigr)^2,\ldots,\\
r'_n-r_n\leq\,&\frac{1}{T(r_n,f)}\leq\frac{1}{T(r'_{n-1},f)}\leq\frac{1}{{(T^{1/2}(r_{n-1},f)+1)}^2}\leq\frac{1}{\bigl(\sqrt{2}+n\bigr)^2}\\
&\,\,\text{and}\,\,T(r_n,f)\geq\bigl(\sqrt{2}+n\bigr)^2,\ldots.
\end{split}
\end{equation*}
As a result, $E'\subseteq\bigcup^\infty_{j=1}[r_j,r'_j]$, and accordingly $\n{E'}\leq\zeta\bigl(2,\sqrt{2}+1\bigr)$.
Here,
\begin{equation*}
\zeta(s,a)=\sum_{n=0}^{\infty}\frac{1}{(n+a)^s}=\frac{1}{a^s}+\frac{1}{(a+1)^s}+\frac{1}{(a+2)^s}+\cdots
\end{equation*}
is the Hurwitz zeta-function confined to $a\in(0,\infty)$ and $s\in(1,\infty)$ with $\zeta(s,1)=\zeta(s)$.
For those $r$ where $1\leq T(r_0,f)\leq T(r,f)<\bigl(\sqrt{2}+1\bigr)^2$, we have $\eqref{Equ12}<\eqref{Equ14}$.
Set $E''$ to be the set of such $r$.
If we include $E''$ as part of the exceptional set in which \eqref{Equ14} fails to be as an upper bound for \eqref{Equ4}, then $E\subseteq\tilde{E}:=E'\cup E''$ and $\n{\tilde{E}}$ may be larger than $\zeta(2)$ or even $2$, depending on the growth of $f(z)$.

To ensure that \eqref{Equ14} be a refinement of \eqref{Equ12} and have a smaller exceptional set, we consider a class of $f(z)$ of infinite order that grow very fast.
Let $f(z)=e^{e^{bz-c}}$ satisfy $T(r,f)=e^{d(r-r_0)}$ for suitable $b,c\in\C\setminus\{0\}$ and $d>0$.
Assume $r'_0-r_0\leq\zeta(2)-\zeta\bigl(2,\sqrt{2}+1\bigr)<1.1334549375$ and $d=\frac{2\ln(\sqrt{2}+1)}{r'_0-r_0}>1.5551982843$.
Thereby, $T(r_0,f)=1$, $T(r'_0,f)=\bigl(\sqrt{2}+1\bigr)^2$, and $\n{\tilde{E}}\leq\n{E'}+\n{E''}\leq\zeta(2)=\frac{\pi^2}{6}<2$.
For instance, select $f(z)$ to satisfy $T(r,f)=e^{1.556(r-1)}$ alongside $r'_0\leq2.134$.
For this meromorphic function $f(z)=e^{e^{bz-c}}$ and infinitely many others of faster growth such as $f(z)=e^{e^{p(z)}}$ for $n\in\N$, \eqref{Equ14} provides a sharper upper bound for \eqref{Equ4} than \eqref{Equ12} on $[r_0,\infty)$, with a refined exceptional set $\tilde{E}$ satisfying $\n{\tilde{E}}\leq\frac{\pi^2}{6}$.
\end{exam}

\section{\sl On a result of Fern\'{a}ndez \'{A}rias}\label{Arias} 
\noindent
Our final result, inspired by Theorem 2 of Fern\'{a}ndez \'{A}rias \cite{aFA1}, refines his result.
As analyzed in Section \ref{BL:gp vs. zeta}, $\log^+T(r,f)$ can represent the least possible $o(T(r,f))$; one can also consider, as Fern\'{a}ndez \'{A}rias did, the other extremal case: $T^\sigma(r,f)$ as the largest possible $o(T(r,f))$ if $0<\sigma<1$.

\begin{lem}\label{HanLiu2}
Each positive continuous increasing function $T(r)$ satisfies
\begin{equation}\label{Equ15}
T(r+\exp(-T(r)))<\exp(T(r))
\end{equation}
outside a possible exceptional set of finite linear measure $S_e:=\sum\limits_{n=0}^\infty a^{-1}_n$, where we write $a_0:=1$, $a_1:=e$, and $a_n:=e^{a_{n-1}}$ recursively for $n\in\N$.
\end{lem}

\begin{proof}
The proof is as before.
Recall $T(r)\geq0$.
Let $E_{\bf{FA}}$ be the closed subset of $(0,\infty)$ in which the inequality \eqref{Equ11} fails to hold.
Inductively, one has
\begin{equation*}
0<r_1<r'_1\leq r_2<r'_2\leq\cdots\leq r_n<r'_n\leq\cdots
\end{equation*}
such that
\begin{equation*}
\begin{split}
r'_1-r_1\leq\,&\frac{1}{\exp(T(r_1))}\leq\frac{1}{e^0}=\frac{1}{a_0}\,\,\text{and}\,\,\exp(T(r_1))\geq a_0,\\
r'_2-r_2\leq\,&\frac{1}{\exp(T(r_2))}\leq\frac{1}{\exp(T(r'_1))}\leq\frac{1}{\exp(\exp(T(r_1)))}\leq\frac{1}{e^{a_0}}=\frac{1}{a_1}\\
&\,\,\text{and}\,\,\exp(T(r_2))\geq a_1,\ldots,\\
r'_{n+1}-r_{n+1}\leq\,&\frac{1}{\exp(T(r_{n+1}))}\leq\frac{1}{\exp(T(r'_n))}\leq\frac{1}{\exp(\exp(T(r_n)))}\leq\frac{1}{e^{a_{n-1}}}=\frac{1}{a_n}\\
&\,\,\text{and}\,\,\exp(T(r_{n+1}))\geq a_n,\ldots.
\end{split}
\end{equation*}
Consequently, $E_{\bf{FA}}\subseteq\bigcup^\infty_{j=1}[r_j,r'_j]$, and accordingly $\n{E_{\bf{FA}}}\leq S_e$.
\end{proof}

\begin{thm}\label{HanLiu3}
Suppose $0<\sigma<1$.
For $T(r)=T^\sigma(r,f)$ and $R=r+\frac{1}{\exp(T^\sigma(r,f))}$, \eqref{Equ15} leads to the following upper bound for \eqref{Equ4} as
\begin{equation}\label{Equ16}
\Bigl(\frac{\sigma+1}{\sigma}\Bigr)\,T^\sigma(r,f)+\log\biggl(1+\frac{1}{r\exp(T^\sigma(r,f))}\biggr),
\end{equation}
outside a possible exceptional set of finite linear measure $S_e$, independent of $\sigma$.
\end{thm}

\begin{proof}
Consider the estimate
\begin{equation*}
T^\sigma(r+\exp(-T^\sigma(r,f)),f)<\exp(T^\sigma(r,f)),
\end{equation*}
which can be easily obtained from \eqref{Equ15}.
The remainder of the derivation for \eqref{Equ16} only involves a routine computation.
\end{proof}

The approximate value of $S_e$ lies in $(1.4338677391,1.4338677392)$.

In fact, a direct computation yields that $S_e(4):=a^{-1}_0+a^{-1}_1+a^{-1}_2+a^{-1}_3+a^{-1}_4$ belongs to $(1.43386773918,1.43386773919)$; so, $S_e>S_e(4)>1.4338677391$.
Put $b_0:=1$, $b_1:=2$, and $b_n:=2^{b_{n-1}}$ recursively with $a_n>b_n$ for $n\in\N$.
$S_e-S_e(4)<\sum\limits^\infty_{n=5}\frac{1}{b_n}<\frac{\gamma(2)}{b_5}=\frac{1}{2^{65535}}<10^{-19728}$; thus, $S_e<S_e(4)+10^{-19728}<1.43386773919+10^{-19728}<1.4338677392$.

\vskip 8pt
\noindent\fbox{{\sf Remark 4}.} Our study of the constant $S_e>0$ is motivated by Borel's work \cite[p. 368]{eB}.
A notable aspect of the preceding result is that the size of the exceptional set $E_{\mathbf{FA}}$ over $(0,\infty)$ is independent of the parameter $\sigma\in(0,1)$.


\bibliographystyle{amsplain}

\end{document}